\documentclass[11pt]{article}
\usepackage[latin1]{inputenc}
\usepackage{epsfig}
\usepackage{color}
\usepackage[british,english]{babel}
\usepackage{amsthm}
\usepackage{amsmath}
\usepackage{amsfonts}
\usepackage{amssymb}
\usepackage{graphicx}
\def\fd{\begin{flushright}$\Box$\end{flushright}}

\newtheorem{corollary}{Corollary}[section]
\newtheorem{defi}[corollary]{Definition}
\newtheorem{example}[corollary]{Example}
\newtheorem{lemma}[corollary]{Lemma}
\newtheorem{prp}[corollary]{Proposition}
\newtheorem{remark}[corollary]{Remark}
\newtheorem{thm}[corollary]{Theorem}
\newfont{\sBlackboard}{msbm10 scaled 900}

\newcommand{\mylabel}[1]{\label{#1}
            \ifx\undefined\stillediting
            \else \fbox{$#1$}\fi }
\newcommand{\BE}{\begin{equation}}

\newcommand{\EEQ}{\end{equation}}
\newcommand{\rfb}[1]{\mbox{\rm
   (\ref{#1})}\ifx\undefined\stillediting\else:\fbox{$#1$}\fi}

\newfont{\Blackboard}{msbm10 scaled 1200}

\newfont{\roma}{cmr10 scaled 1200}

\def\CC{\rm \hbox{C\kern-.56em\raise.4ex
         \hbox{$\scriptscriptstyle |$}\kern+0.5 em }}

\newcommand{\mm}    {{\hbox{\hskip 0.5pt}}}

\newcommand{\bluff} {{\hbox{\raise 15pt \hbox{\mm}}}}
scaled\magstep2

\makeatletter
\def\section{\@startsection {section}{1}{\z@}{-3.5ex plus -1ex minus
    -.2ex}{2.3ex plus .2ex}{\large\bf}}
\makeatother
\def\be{\begin{equation}}
\def\ee{\end{equation}}

\title{\textbf { Grushin problems and control theory: Formulation and examples}}
\author{  Akram Ben Aissa\\
\\
\small Département de Mathématiques, Faculté des Sciences de Monastir,\\ \small Université de Monastir, 5019 Monastir, Tunisie\\
\small Unité de recherche: Analyse et Contrôle des Equations aux Dérivées Partielles\\
\small -ACEDP- 05/UR/15-01.\\
\small \texttt{issaakram26@gmail.com}}

\date{}
\begin{document}
\maketitle
 \abstract{}
 In this paper we give a new formulation of an abstract control
 problem in terms of a Grushin problem, so that we will reformulate all
 notions of controllability, observability and stability in a new
 form that gives readers an easy interpretation  of these notions.


\section{Introduction}
Grushin problem is a simple linear algebraic tool which has proved
itself very useful in the mathematical study of spectral problems
arising in electromagnetism and quantum mechanics.\\
This approach appears constantly under different names and guises in
many works of pure and applied mathematics.\\
The key observation goes back to Schur and his complement formula:\\
If we have for matrices
$$\left[
\begin{array}{cc}
  P & R_- \\
  R_+ & 0 \\
\end{array}
\right]^{-1}=\left[
\begin{array}{cc}
  E & E_+ \\
  E_- & E_{-+} \\
\end{array}
\right],$$ then $P$ is invertible if and only if $E_{-+}$ is
invertible and
$$P^{-1}=E-E_+E^{-1}_{-+}E_-,\;E^{-1}_{-+}=-R_+P^{-1}R_-.$$
This tools was developed by J. Sj\"{o}strand and M. Zworski
\cite{SjZ}, Hager and Sj\"{o}strand \cite{HSj1}, Hellfer and
Sj\"{o}strand \cite{HSj2}.\\
The aim of this paper is to reformulate abstract control problems
studied in control theory by Weiss \cite{G1} and Ammari and Tucsnak
\cite{AT} in a form of Grushin problems and give some regularity
results arising in the two theory.\\
More concisely, let $U,\,X$ be two Hilbert spaces and consider the
abstract control problem
\begin{equation}\label{11}
\left\{
\begin{array}{ll}
\dot{z}(t)=Az(t)+Bu(t),\quad z(0)=z_0\\
y(t)=B^*z(t)\\
\end{array}
\right.
\end{equation}
where $A:D(A)\subset X\longrightarrow X$ generates a
$C_0$-semigroups of contractions $T(t)_{t\geq 0}$,
$B\in\mathcal{L}(U,X)$ is an admissible control operator, $u\in
L^2_{\text{loc}}(0,+\infty; U)$. The transfer function of (\ref{11})
is given by $H(\lambda)\in\mathcal{L}(U)$ such that
$$\hat{y}(\lambda)=H(\lambda)\hat{u}(\lambda),$$
where $\hat{}$ denotes the Laplace transformation . For these
concepts, see \cite{TW}.\\
 Suppose that $H(\lambda)$ is invertible in  $\mathcal{L}(U)$, therefore system (\ref{11})
 can be written as a well-posed Grushin problem as:\\
 for $\lambda\in\rho(A)$
 \begin{equation}\label{22}
 \left\{
 \begin{array}{ll}
 (\lambda-A)u+Bu_-&=v\\
 B^*u&=v_+\\
 \end{array}
 \right..
 \end{equation}
 Thus, (\ref{22}) is well-posed if
 $$\left[
\begin{array}{cc}
  \lambda-A & B \\
  B^* & 0 \\
\end{array}
\right]^{-1}=\left[
\begin{array}{cc}
  E & E+ \\
  E_- & E_{-+} \\
\end{array}
\right],$$ we refer to $E_{-+}$ as the effective Hamiltonien of
$\lambda-A$. We prove that the inverse of the transfer function of
system (\ref{11}) is the effective Hamiltonien of $\lambda-A$ in
(\ref{22}).

 The paper is organized as follows. In the second section we
give some preliminary results dealing to system theory, and we
investigate some spectral properties of transfer function, moreover
 we show how regularity property (in the Weiss sense)  of system (\ref{11}) is stable under iterations of Grushin
 problems. Our main results and statements are given in section $3$. The last
section is devoted to some applications.
\section{Some background}
In this section we gather, for easy reference, some basic facts
about admissible control and observation operators, about well-posed
and regular linear systems, their transfer functions, well-posed
triples of operators and closed-loop systems. For proofs and for
more details we refer to the literature.\\
We assume that $X$  is a Hilbert space and $A : D(A)\longrightarrow
X$ is the generator of a strongly continuous semigroup $\mathbb{T}$
on $X$. We define the Hilbert space $X_1$ as $D(A)$ with the norm
$\|z\|_1 = \|(\beta I-A)z\|$, where $\beta\in \rho(A)$ is fixed
(this norm is equivalent to the graph norm). The Hilbert space
$X_{-1}$ is the completion of $X$ with respect to the norm
$\|z\|_{-1} = \|(\beta I-A)^{-1}z\|$. This space is isomorphic to
$D(A^*)^*$, and we have
$$ X_1\subset X \subset X_{-1},$$
densely and with continuous embeddings. $\mathbb{T}$ extends to a
semigroup on $X_{-1}$, denoted by the same symbol. The generator of
this extended semigroup is an extension of $A$, whose domain is $X$,
so that $A : X\longrightarrow X_{-1}$. We assume that $U$ is a
Hilbert space and $B\in\mathcal{L}(U, X_{-1})$  is an admissible
control operator for $\mathbb{T}$, defined as in Weiss \cite{G3}.
This means that if $z$  is the solution of $z(t) = Az(t) + Bu(t)$,
which is an equation in $X_{-1}$, with $z(0) = z_0\in  X$ and $u \in
L^2(\mathbb{R}_+, U)$, then $z(t)\in  X$ for all $t \geq 0$. In this
case, $z$ is a continuous $X-$valued function of $t$. We have
\begin{equation}
z(t)=\mathbb{T}_t+c(t)u,
\end{equation}
where $c(t) \in \mathcal{L}(L^2(\mathbb{R}_+, U);X)$ is defined by
\begin{equation}\label{opc}
c(t)u=\int_0^t \mathbb{T}_{t-s}Bu(s)ds.
\end{equation}
The above integration is done in $X_{-1}$, but the result is in $X$.
The Laplace transform of $z$ is
$$ \hat{z}(s)=(sI-A)^{-1}[z_0+ B\hat{u}(s)].$$
$B$ is called bounded if $B \in \mathcal{L}(U, X)$ (and unbounded
otherwise). If $B$ is an admissible control operator for
$\mathbb{T}$, then $(sI-A)^{-1}B\in\mathcal{L}(U, X)$ for all $s$
with  $\mathfrak{Re}(s)$  sufficiently large. Moreover, there exist
positive constants $\delta,\,\omega$  such that
$$\|(sI-A)^{-1}B\|_{\mathcal{L}(U, X)}\leq \frac{\delta}{\sqrt{Re
s}},\quad \forall Re s>\,\omega,$$ and if $\mathbb{T} $ is normal
then the last inequality  implies admissibility, see \cite{G1}.\\
We assume that Y is another Hilbert space and $C \in\mathcal{L}(X_1,
Y )$ is an admissible observation operator for $\mathbb{T}$, defined
as in Weiss \cite{G4}. This means that for every $T > 0$  there
exists a $K_T\geq  0$ such that
\begin{equation}
\int_0^T \|C\mathbb{T}_t z_0\|^2dt\leq
K_T^2\|z_0\|^2\quad\forall\,z_0\in D(A).
\end{equation}
$C$ is called bounded if it can be extended such that $C
\in\mathcal{L}(X, Y )$.\\
We regard $L^2_ {loc}(\mathbb{R}_+; Y)$  as a Fréchet space with the
seminorms being the $L^2$ norms on the intervals $[0, n],\;
n\in\mathbb{N}$. Then the admissibility of $C$ means that there is a
continuous operator $\Psi : X \longrightarrow L^2_{loc}([0,\infty),
Y)$ such that
\begin{equation}\label{obser}
(\Psi z_0)(t)=C\mathbb{T}_tz_0\quad \forall\,z_0\in D(A).
\end{equation}
The operator $\Psi$    is completely determined by (\ref{obser}),
because $D(A)$ is dense in $X$. Now we introduce two extensions of
$C$ as following:
\begin{defi} Let $X$ and $Y$ be Hilbert spaces with $\mathbb{T}$ a $C_o$-semigroup on $X$ and
suppose that $C \in\mathcal{L}(X_1, Y)$. Then the Lebesgue extension
of $C$ (with respect to $\mathbb{T}$), $C_L:D(C_L)\longrightarrow Y$
defined by
\begin{equation}\label{lebeg}
C_Lx=\displaystyle{\lim _{t\rightarrow 0}}\,C\frac{1}{t}\int_0^t
\mathbb{T}_sxds
\end{equation}
with $D(C_L)=\{x\in X|\text{the limit in (\ref{lebeg})}\,
\text{exists}\}$.
\end{defi}
Weiss showed in \cite{G4} that $C_L$ is an extension of $C$, in
particular,
$$X_1 \hookrightarrow D(C_L) \hookrightarrow X.$$
The significance of the Lebesgue extension, $C_L$, is that it makes
it possible to give a simple pointwise interpretation of the output
map $(\ref{obser})$ for every $x$ in the original state space $X$.
For every $x_0\in X$, there holds $\mathbb{T}_tx_0\in D(C_L)$ for
almost every $t\geq 0$ and
$$(\Psi x_0)(t)=C_L\mathbb{T}_tx_0.$$
A similar $\Lambda$-extension of $C$  was introduced by Weiss
\cite{G1}
\begin{equation}\label{gamma}
C_{\Lambda} x_0=\displaystyle{\lim _{\lambda\rightarrow
+\infty}}\,C\lambda(\lambda I-A)^{-1}x_0,
\end{equation}
for $\lambda \in \mathbb{C}$ with $\mathfrak{Re}(\lambda)$
sufficiently large and for all $x_0$ in $D(C_\Lambda)=\{x_0\in
X|\,\text{the limit in} \eqref{gamma}\,\text{exists}\}$.
\begin{defi}\label{fonction}
Let $U,\, X,\, Y,\, V \;\text{and}\;W$ be Hilbert spaces such that
$W \subset X \subset V$ and let $B \in\mathcal{L}(U, V)$ and $C
\in\mathcal{L}(W, Y)$ and let $\mathbb{T} = (\mathbb{T}_{t})_{t\geq
0}$ be a $C_0$-semigroup on $X$. Suppose that $B$ is an admissible
control operator for $\mathbb{T}$ with respect to $V$ and that $C$
is an admissible observation operator for $\mathbb{T}$ with respect
to $W$. Then we define the transfer functions of the triple
$(A,B,C)$ to be the solutions, $H:\rho(A)\longrightarrow
\mathcal{L}(U,Y)$ of
\begin{equation}\label{transfer}
\frac{H(s)-H(\beta)}{s-\beta}=-C(sI-A)^{-1}(\beta I-A)^{-1}B
\end{equation}
for $s,\beta\,\in\rho(A),\,s\neq\beta$, where $\rho(A)$ is the
resolvent set of $A$.
\end{defi}
We remark that, since $B$ is an admissible control operator for
$\mathbb{T},\, (\beta I - A)^{-1}B$ is an $\mathcal{L}(U, X)$-valued
analytic function and since $C$ is an admissible observation
operator for $\mathbb{T},\, C(sI-A)^{-1}$ is a $\mathcal{L}(X,
Y)$-valued analytic function. Both $(\beta I-A)^{-1}B$ and
$C(sI-A)^{-1}$  are analytic on some right half-plane
$\mathbb{C}_{\alpha} = \{s\in\mathbb{C} : Re(s)
> \alpha\}$. Consequently the transfer functions always exist as $\mathcal{L}(U,
Y)$-valued functions which are analytic in some
$\mathbb{C}_{\alpha}$. They differ only by an additive constant,
$D\in\mathcal{L}(U, Y)$ (often called feedthrought operator). The
point is that they don't need  necessarily be bounded on any
$\mathbb{C}_{\alpha}$. We impose this as an extra assumption on the
triple $(A, B, C)$  and call this well-posedness.
\begin{defi}
Under the same assumptions as in Definition (\ref{fonction}), we say
that the triple $(A, B, C)$ is well-posed if $B$ is an admissible
control operator for $\mathbb{T}$ with respect to the Hilbert space
$V$, $C$  is an admissible observation operator for $\mathbb{T}$
with respect to the Hilbert space $W$ and its transfer function is
bounded on some half-plane $\mathbb{C}_{\alpha}$.
\end{defi}
\vskip 0.15cm Next, we give some notions of controllability and
observability. For more details, see \cite{TW}.\\
 Let $A:\mathcal{D}(A)\longrightarrow X$ generates a
$C_0$-semigroup $\mathbb{T}_t$ on $X$, $B\in\mathcal{L}(U,X)$, and
$z_0\in X$.
\begin{defi}(Controllability)
The  system  $(A,B)$ is said to be \textbf{exactly controllable} in
time $T>0$ if for every $z_0, z_1\in X$ there exists $u\in L^2(0,T;
U)$ such that the solution of the system (A,B) given by the Duhammel
formula verify $z(T)=z_1$.
\end{defi}
The fact that (A,B) is exactly controllable in  $T>0$ is equivalent
to the fact that the operator $c(t)$ defined by (\ref{opc})   is
surjective, that's
$$\mathrm{Im}\,c(t)=X.$$

\begin{defi}(Observability)
Let $A$ be a generator of $C_0$-semigroup $\mathbb{T}_t$ on $X$, and
$C\in\mathcal{L}(X, U)$. The system  (A, C) is said to be
\textbf{exactly observable} in time $T>0$ if there exists $\delta>0$
such that
\begin{equation}\label{obinequality}
\int_0^T\|C\mathbb{T}_tz\|^2_{U}dt\geq \delta\|z\|^2_{X},\quad
\forall z\in X.
\end{equation}
\end{defi}
\begin{remark}\label{rems}
For every $T>0$, we denote by
\begin{equation}
(\Psi_Tz)(t)=\left\{
\begin{array}{ll}
C\mathbb{T}_tz\quad t\in[0, T]\\
0\quad \qquad t>T.\\
\end{array}\right.
\end{equation}
Since $C$ is bounded, then $\Psi_T\in\mathcal{L}(X, L^2((0,\infty),
U))$ for every $T>0$, and we remark that (A, C) is exactly
observable in time $T>0$ if and only if there exists $\delta>0$ such
that $$\|\Psi_Tz\|_{L^2(0,\infty; U)}\geq \delta\|z\|_X\quad \forall
z\in X.$$
\end{remark}
The following theorem gives the links between these concepts.
\begin{thm}
Let $A$ be a generator of a semigroup on $X$ and $B\in\mathcal{L}(U,
X)$. Then, the following assertions are equivalent :
\begin{enumerate}
 \item $(A, B)$ is exactly controllable on $[0,T]$.
 \item $(A^*, B^*)$ is exactly observable in time $T>0$.
 \end{enumerate}
\end{thm}
\begin{proof}
\vskip 0.1cm  We set $c(t)u:=\int_0^t\mathbb{T}_{t-s}Bu(s)ds$. Then,
(A, B) is exactly controllable if and only if $\mathrm{Im}\,c(T)=X$,
which is equivalent to saying that $c(T)^*$ is bounded below, i.e.,
there exists $\delta>0$ such that
\begin{equation}\label{1}
\|c(T)^*z\|\geq \delta\|z\|,\quad \forall\:z\in X.
\end{equation}
Compute $c(T)^*$. For all $u\in L^2((0,\infty),U)$ and $z\in X$:
\begin{equation*}
\begin{split}
\langle c(T)u,z\rangle&=\langle\int_0^T\mathbb{T}_{T-s}Bu(s)ds,z\rangle_X\\
&=\int_0^T\langle u(s),B^*\mathbb{T}^*_{T-s}z\rangle_Uds\\
&=\langle u,\Lambda_T\Psi^d_Tz\rangle_{L^2((0,\infty),U)}
\end{split}
\end{equation*}
where $\Lambda_Tu(t):=\left\{
\begin{array}{ll}
u(T-t)\quad t\in [0,T]\\
0\quad t>T\\
\end{array}
\right.$ for all $u\in L^2((0,\infty),U)$.\\\vskip 0.9cm Thus,
$c(T)^*=\Lambda_T\Psi^d_T$,  and inequality (\ref{1}) becomes
$$\|\Psi^d_Tz\|\geq \|\Lambda_T\Psi^d_Tz\|\geq\delta\|z\|\quad\forall\;z\in X$$
since $\Lambda_T$ is un unitary. And therefore, $(A^*,B^*)$ is
exactly observable by Remark \ref{rems}.
\end{proof}
 Now, we introduce the notion of Grushin problem and Schur
Complement.
\begin{defi}
Let $P:H_1\longrightarrow H_2$ be a linear operator where
$H_1,\;H_2$ are two Hilbert spaces. Then, a \textbf{Grushin problem}
for $P$ is a system
\begin{equation}
\left\{
\begin{array}{ll}
Pu+R_-u_-=v\\
R_+u=v_+\\
\end{array}\right.
\end{equation}
where $R_-:H_-\longrightarrow H_2,\;R_+:H_1\longrightarrow
H_+,\,(u,u_-)\in H_1\times H_-$ are unknown and $(v,v_+)\in
H_2\times H_+$ are given. In matrix form we can write
$$\mathcal{P}:=\left[
\begin{array}{ccc}
P&&R_-\\
R_+&&0\\
\end{array}
\right]:H_1\oplus H_-\longrightarrow H_2\oplus H_+.$$ We say that
the Grushin problem is well posed if we have the inverse
$$\mathcal{E}=\left[
\begin{array}{ccc}
E&&E_+\\
E_-&&E_{-+}\\
\end{array}\right]: H_2\oplus H_+\longrightarrow H_1\oplus H_-.$$
In this case we will refer to $E_{-+}$ as \textbf{the effective
Hamiltonian} of $P$.
\end{defi}
 For the concepts of Grushin problems and Schur
complements we refer readers to \cite{G}, \cite{HSj1},\cite{HSj2}.

\section{Reformulation of abstract control problem}
We will connect the theory of well-posed linear system with
well-posed Grushin problems.\\
 Let $A:D(A)\subset X\longrightarrow X $ skew-adjoint and then generates $C_0$-group of isometries  on $X$  and $B\in\mathcal{L}
 (U, X)$ where $U$ is another Hilbert space identified with its dual. Thus, we prove that
\begin{prp}

If the following  abstract control problem with observation (or
equivalently the triple of operators $(A, B, B^*)$)
\begin{equation}\label{cp}
\left\{
\begin{array}{lll}
\dot{z}(t)=Az(t)+Bu(t),\\
y(t)=B^*z(t),\\
z(0)=z_0\\
\end{array}\right.
\end{equation}
 is well-posed with $A$   as above and $B\in\mathcal{L}(U, X)$ with invertible
transfer function $H(\lambda)$  then,  the following Gru\v{s}hin
problem is well-posed:\\
For all $(v,v_+)\in X\times U$, there exists  $(u,u_-)\in X\times U$
solution of
\begin{equation}\label{gp}
\left\{
\begin{array}{lll}
(\lambda I-A)u+Bu_-&=&v\\
\qquad B^*u&=&v_+\\
\end{array}\right.
\end{equation}

and  the effective  Hamiltonien is given by
\begin{equation}\label{he}
E_{-+}(\lambda)^{-1}=-B^*(\lambda I-A)^{-1}B=-H(\lambda).
\end{equation}
 Note that the transfer function is define as $\hat{y}(\lambda)=H(\lambda)\hat{u}(\lambda)$, where $\hat{}$ denotes the Laplace transform with  $z_0=0$.
\end{prp}
\begin{proof}
 Suppose that (\ref{cp}) is well-posed and the associated transfer
function $H(\lambda)\in\mathcal{L}(U)$ is invertible. Since
$(\lambda-A)^{-1}B$ takes its values in $X=D(B^*)$ for all
$\lambda\in\rho(A)$, then $H(\lambda)$ is given explicitly by the
desired formula
\begin{equation}\label{dtf}
H(\lambda)=B^*(\lambda-A)^{-1}B,\quad \forall\lambda\in\rho(A).
\end{equation}
 For given $(v,v_+)\in X\times U ;\text{and}\;\lambda\in\rho(A)$, is
there $(u,u_-)\in X\times U$ such that
\begin{equation*}
\left\{
\begin{array}{lll}
(\lambda I-A)u+Bu_-&=&v\\
\qquad B^*u&=&v_+\\
\end{array}\right.
\end{equation*}
from the first equation, we can write
$$u=(\lambda-A)^{-1}v-(\lambda-A)^{-1}Bu_-,$$
and the second equation becomes
$$H(\lambda)u_-=B^*(\lambda-A)^{-1}v-v_+,$$
and therefore, since $H(\lambda)$ is invertible  we get that
(\ref{gp}) is well-posed.
\end{proof}\fd
Now, we consider abstract control problems with feedthrought
operators $D\neq 0$, in the form
\begin{equation}\label{ab23}
\left\{
\begin{array}{lll}
\dot{z}(t)=Az(t)+Bu(t),\\
y(t)=B^*z(t)+ Du(t),\\
z(0)=z_0\\
\end{array}\right.
\end{equation}
where $A:D(A)\subset X\longrightarrow X$, skew-adjoint and
$B\in\mathcal{L}(U, X)\;\text{and} \;D\in\mathcal{L}(U)$. It was
showed in Weiss \cite{G1}, that with $z_0=0$ and
$\mathfrak{Re}(\lambda)$ sufficiently large, the transfer function
of (\ref{ab23}) is given by
\begin{equation}\label{tf1}
H(\lambda)=D+B^*(\lambda I-A)^{-1}B.
\end{equation}
and it satisfies the equation
\begin{equation}\label{pft}
\frac{H(s)-H(\beta)}{s-\beta}=-C(sI-A)^{-1}(\beta I-A)^{-1}B,
\end{equation} for any $s,\beta\in\rho(A)$ with $s\neq \beta$. Thus, the
connection between the transfer function and the effectif
Hamiltonien is given by the following lemma.
\begin{lemma}
for $\lambda\in\rho(A)$, suppose that the following Grushin problem
is well-posed:\\
for given $(v,v_+)\in X_{-1}\times U$, there exists unique
$(u,u_-)\in D(A)\times U$ such that
\begin{equation}\label{grp1}
\left\{
\begin{array}{ll}
(\lambda I-A)u +&Bu_-=v,\\
B^*u \,+&Du_-=v_+\\
\end{array}\right.
\end{equation}
therefore, the effectif Hamiltonien of $(\lambda I-A)$ is given by
\begin{equation}
E_{-+}(\lambda)^{-1}=D-B^*(\lambda I-A)^{-1}B.
\end{equation}
in the sense that its invertibility controls the existence of the
resolvent.
\end{lemma}
In fact, feedthrought operator play an important  role in the study
of regularity of such abstract control problems with observation in
the Weiss sense as we will see in  section 3, thus we have the
characterization of regularity obtained in Weiss \cite{G1}:
\begin{thm}[ Weiss \cite{G1}]
 An abstract linear system is regular if and only if its
transfer function has a strong limit at $+\infty$ (along the real
axis), and we have
$$\displaystyle{\lim _{\lambda\in\mathbb{R},\lambda\rightarrow
+\infty}}\, H(\lambda)v=Dv,\quad \forall\,v\in U.$$
\end{thm}
Let us now consider  abstract  problems (\ref{cp})  with unbounded
control and observation operators, that's with the same assumption
on $A$ and $\tilde{A}$  is an extension of $A\;\text{with
domain}\;D(A)$ on $X_{-1}$ denoted also by $A$ exception that
$B\in\mathcal{L}(U, X_{-1})$ assumed to be admissible and $U$ is an
Hilbert space identified with its dual. We assume that its transfer
function  is invertible as an element of $\mathcal{L}(U)$(here we
have not explicitly its desired expression (\ref{dtf}) but the only
thing we know that it checks  relation (\ref{pft})).

Under a suitable construction of a well-posed Grushin problem, we
prove some properties of transfer function of (\ref{cp}).
\begin{prp}
Let $\mathcal{O}_c$ be a connected open of $\rho(A)$ and
$\lambda\in\mathcal{O}_c$.\\
Then, $(H(\lambda))_{\lambda\in\mathcal{O}_c}$ is a family of
Fredholm operators depends holomorphically on $\lambda$. Moreover,
if $H(\lambda_0)^{-1}$ exists at some points
$\lambda_0\in\mathcal{O}_c$, then
$\mathcal{O}_c\ni\lambda\longrightarrow H(\lambda)^{-1}$ is
meromorphic.
\end{prp}
\begin{proof}
From the reformulation of abstract control problem with observation
on a well-posed Grushin problem, with $B^*_L$ the Lebesgue extension
of $B$ in the place of this later,  we proved that the link is the
invertibility of the family of transfer functions
$(H(\lambda)_{\lambda\in\mathcal{O}_c}$ where $\mathcal{O}_c$ is a
connected open of $\rho(A)$, and consequently is Fredholm of index
$0$.\\
For $\lambda_0\in\mathcal{O}_c\subset \rho(A)$ (and therefore
$\lambda_0-A$ is Fredholm), we can always take $U$ with finite
dimensional.\\
 let $n_+=\dim\mathrm{ker}(\lambda_0-A)=\dim\mathrm{coker}(\lambda_0-A),\quad
n_+=n_-=n$ and choose  $B:\mathbb{C}^{n}\longrightarrow X$.  In this
case
$$E^{\lambda_0}_{-+}=H(\lambda_0)^{-1}:\mathbb{C}^n\longrightarrow
\mathbb{C}^n$$ is finite matrix with index $n_+-n_-=0$.\\
The invertibility of $H(\lambda_0),\;\lambda_0\in\mathcal{O}_C$ is
equivalent to the well-posedness of $\mathcal{A}(\lambda)$ where
$$\mathcal{A}(\lambda)=\left[
\begin{array}{cc}
  \lambda-A & B \\
  B^* & 0 \\
\end{array}
\right].$$ This shows that there exists a locally finite covering of
$\mathcal{O}_c$, $\{O_j\}$, such that for $\lambda\in
O_j,\;H(\lambda)$ is invertible, more precisely when
$f_j(\lambda)\neq 0$, where $f_j$ is holomorphic in $O_j$(indeed, we
can define $f_j(\lambda_0):=\det\,H(\lambda_0)^{-1}$ where
$H(\lambda_0)^{-1}$ exists for $\lambda_0\in O_j$). Since $O_c$ is
connected and since $H(\lambda_1)$ is invertible for at least one
$\lambda_1\in\mathcal{O}_c$ shows that all $f_j$'s are not
identically zero. That means that $\det\,H(\lambda)^{-1}$ is
non-vanishing holomorphic in some neighbourhood of $\lambda_0$,
$V(\lambda_0)$, and consequently $H(\lambda)$ is a family of
meromorphic operators in $V(\lambda_0)$, where $\lambda_0$ was
arbitrary in $\mathcal{O}_c$.

\end{proof}

\begin{prp}
Let $g$ be holomorphic function on $\mathcal{O}_c$ connected open of
$\rho(A)$. Then for any curve $ \gamma$ homologous to $0$ in
$\mathcal{O}_c$, and on which $(\lambda-A)^{-1}$ exists,  the
operator $\frac{1}{2\pi
i}\int_{\gamma}(\lambda-A)^{-1}g(\lambda)d\lambda$ (that's the
spectral projection of $A$ onto $\mathcal{O}_c$ ) is of trace class
and we have
\begin{equation}\label{tracef}
\mathrm{tr}
\int_{\gamma}(\lambda-A)^{-1}g(\lambda)d\lambda=\mathrm{tr}\int_{\gamma}\partial_{\lambda}H(\lambda)^{-1}H(\lambda
)g(\lambda)d\lambda.
\end{equation}
\end{prp}
\begin{proof}
Basic idea: writing $\partial_{\lambda}A(\lambda)=\dot{A}(\lambda)$,
we have
$$\dot{\mathcal{E}}(\lambda)=-\mathcal{E}(\lambda)\dot{\mathcal{A}}(\lambda)\mathcal{E}(\lambda)$$
where $\mathcal{E}(\lambda)$ as in the previous proposition and
$\mathcal{E}(\lambda)$ is given by
$$\left[
\begin{array}{ccc}
  E(\lambda) & &E_+(\lambda) \\
  E_-(\lambda)& & E_{-+}(\lambda) \\
\end{array}
\right]$$which gives
$$E_-(\lambda)E_+(\lambda)=-\dot{\mathcal{E}}_{-+}(\lambda)$$
we recall that
$$(\lambda-A)^{-1}=E(\lambda)-E_+(\lambda)\dot{\mathcal{E}}_{-+}(\lambda)E_-(\lambda).$$
Since $E_{-+}(\lambda)^{-1}$ is a finite matrix, then
$$
\int_{\gamma}(\lambda-A)^{-1}g(\lambda)d\lambda=-
\int_{\gamma}E_+(\lambda)E_{-+}(\lambda)^{-1}E_-(\lambda)g(\lambda)d\lambda$$
is an operator of trace class.

\end{proof}


\subsection{Some regularity results}
In this section we show how the property of regularity in the Weiss sense \cite{G1} is conserved along the iterations of Grushin problems.\\
Consider the system of evolution equations
\begin{equation}\label{evo}
\left\{
\begin{array}{ll}
\dot{z}(t)&=Az(t)+Bu(t),\quad z(0)=z_0,\\
y(t)&=B^*z(t)\\
\end{array}\right.
\end{equation}
where
\begin{enumerate}
\item $A:D(A)(\subset X_{-1})\longrightarrow X_{-1}$ is an unbounded positive self-adjoint operator in the
Hilbert space $X$,\\
\item $B\in\mathcal{L}(U,X_{-1})$,\\
\item $B^*\in\mathcal{L}(X_1, U)$ is defined as
$$(B^*x,u)_U=\langle x,Bu\rangle_{X_1\times X_{-1}}\quad \forall x\in X_1.$$
\end{enumerate}
Assume that (\ref{evo}) is well-posed and its transfer function
$H(s)\in\mathcal{L}(U)$  is uniquely determined by the pair $(A,B)$
and assumed to be invertible and checks the following relation
\begin{equation*}
 \frac{H(\lambda)-H(\mu)}{\lambda-\mu}=-B^*(\lambda I-A)^{-1}(\mu
 I-A)^{-1}B,\quad\forall\lambda,\,\mu\in\rho(A).
\end{equation*}
Suppose that system (\ref{evo}) is  regular; that's:
$$\lim _{\lambda\in\mathbb{R},\lambda\rightarrow
+\infty}\,H(\lambda)u=Du\quad \forall\,u\in U,$$ where
$D\in\mathcal{L}(U)$ called feedthrought operator. For more details
we refer to \cite{C1},\cite{G1}....Thus, for $\lambda \in\rho(A)$,
the associated Grushin problem is
\begin{equation}\label{gp1}
\left\{
\begin{array}{ll}
(\lambda I-A)u_1+Bu_2&=v_1\\
B^*u_1&=v_2.\\
\end{array}\right.
\end{equation}
In matrix form, (\ref{gp1}) is written as
$$\mathcal{A}(\lambda)=\left[
\begin{array}{ccc}
  \lambda I-A &  & B \\
  B^* &  & 0 \\
\end{array}
\right]:X_1\oplus U\longrightarrow X_{-1}\oplus U.$$ Hence,
(\ref{gp1}) is well-posed if and only if $\mathcal{A}(\lambda)$ is
invertible with
$$\mathcal{A}(\lambda)^{-1}=\left[
\begin{array}{ccc}
  E(\lambda) &  & E_+(\lambda) \\
  E_-(\lambda) &  & E_{-+}(\lambda) \\
\end{array}
\right].$$ In the Grushin problem context, $E_{-+}(\lambda)$ is
called the effective Hamiltonien of $(\lambda I-A)$, and is also the
Schur complement of $(\lambda I-A)$ and we have
$$E_{-+}(\lambda)^{-1}=-B^*(\lambda I-A)^{-1}B,\quad\forall\lambda\in\rho(A)$$ which is invertible.
System (\ref{gp1}) can be iterated in the following way:\\
Assume that there exists two operators
$$N_-:V_-\longrightarrow U,\quad N_+:U\longrightarrow V_+,$$ with
$V_-,\,V_+$ are two Hilbert space such that the following Grushin
problem is well-posed
\begin{equation}\label{gp2}
\left\{
\begin{array}{ll}
E_{-+}(\lambda)u_3+N_-u_4&=v_3\\
N_+u_3&=v_4\\
\end{array}\right.
\end{equation}
that's
$$\mathcal{E}=\left[
\begin{array}{ccc}
  E_{-+}(\lambda) &  & N_- \\
  N_+ &  & 0 \\
\end{array}
\right]:U\oplus V_-\longrightarrow U\oplus V_+$$ is invertible with
the inverse
$$\mathcal{F}=\left[
\begin{array}{ccc}
  F(\lambda) &  & F_+(\lambda) \\
  F_-(\lambda) &  & F_{-+}(\lambda) \\
\end{array}
\right],$$ then the new Grushin problem
\begin{equation}\label{gp3}
\left\{
\begin{array}{ll}
(\lambda I-A)u+BN_-\tilde{u}&=\tilde{v}\\
N_+B^*u&=\tilde{v}_-\\
\end{array}\right.
\end{equation}
with the inverse given by
$$\mathcal{G}=\left[
\begin{array}{ccc}
  E-E_+FE_- &  &E_+F_+  \\
  F_-E_- &  &-F_{-+}(\lambda)  \\
\end{array}
\right].$$ Thus, the corresponding evolution problem to (\ref{gp3})
is
\begin{equation}\label{e2}
\left\{
\begin{array}{ll}
\dot{z}(t)=(\lambda I-A)z(t)+BN_-v(t),\quad z(0)=z_1\\
y_2(t)=N_+B^*z(t)\\
\end{array}\right.
\end{equation} which still  regular with transfer function given
$$H_1(\lambda)=N_+H(\lambda)N_-.$$

\section{Application of Grushin problem in control theory}
Let us  starting  by recalling some definitions and properties
mentioned in \cite{CWW}.
\begin{defi}
Suppose that $\Sigma=(\mathbb{T},\Phi,\mathbb{L},\mathbb{F})$ is an
abstract linear system. If $A$ is the generator of $\mathbb{T}$, $B$
is the control operator of $\Sigma$ and $C$ is the observation
operator of $\Sigma$, then we say that $(A,B,C)$ is the triple
associated with $\Sigma$. A triple of operators $(A,B,C)$ will be
called well-posed if there is an abstract linear system $\Sigma$
such that $(A,B,C)$ is the triple associated with $\Sigma$.
\end{defi}
In the following two remarks we try to clarify what well-posedness
of triple of operators  means in terms of differential equations.
See \cite{CWW}.
\begin{remark}
Suppose that $U,X$ and $Y$ are Hilbert  spaces, $A$ is the generator
of a semigroup on $X$, $B\in\mathcal{L}(U,X_{-1})$ and
$C\in\mathcal{L}(X_1,Y)$. If $C_L$ is the Lebesgue extension of $C$,
and if the operator $C_L(\beta I-A)^{-1}B$ is well defined for some
(and  hence any) $\beta\in\rho(A)$, then $(A,B,C)$ is well-posed if
and only if the system of equations
\begin{equation}\label{cp11}
\left\{
\begin{array}{ll}
\dot{z}(t)=Az(t)+Bu(t),\quad z(0)=0\\
y(t)=C_Lz(t)\\
\end{array}\right.
\end{equation} is well-posed in a certain  natural sense. If the
triple is well-posed, but $C_L(\beta I-A)^{-1}B$ does not exist,then
(\ref{cp}) is no longer well-posed.
\end{remark}
\begin{remark}
Let $U,X,Y,A,B,C$ and $C_L$ be as in the previous remark, but we do
not assume that $C_L(\beta I-A)^{-1}B$ makes sense. Then $(A,B,C)$
is well-posed if and only if the following (more complicated) system
of equations is well-posed:
\begin{equation*}
\left\{
\begin{array}{cc}
\dot{z}(t)=Az(t)+Bu(t),\quad z(0)=0\\
y(t)=C_L[z(t)-(\beta I-A)^{-1}Bu(t)]\\
\end{array}\right.
\end{equation*}
in the same sense as (\ref{cp}).
\end{remark}
In this section, we show how a well-posed Grushin problem of type
(\ref{gp}) gives a Hautus test Criteria and then exact observability
and exponential stability of system of type (\ref{cp}).\\
Before starting, we recall some properties and definitions, for more
details see Miller \cite{M} and Hautus \cite{H}. The exact
observability property is dual to the exact controllability
property, as it has been shown in Dolecki and
Russell [8].\\
 few papers in the area of controllability and
observability of systems governed by partial differential equations
have considered a frequency domain approach, related to the
classical Hautus test in the theory of finite dimensional systems
(see Hautus \cite{H}). Roughly speaking, a frequency domain test for
the observability of (\ref{cp}) is formulated only in terms of the
operators $A,\, B^*$ and of a parameter (the frequency). This means
that the time $t$ does not appear in such a test and that we do not
have to solve an evolution equation. In the case of a bounded
observation operator $B^*$, such frequency domain methods have been
proposed in Liu \cite{K}. In the case of an unbounded observation
operator $B^*$,  a Hautus type test has been recently obtained in
Miller \cite{M}. Thus we have
\begin{prp}
The system (\ref{evo}) is exactly observable in time $T>0$ if and
only if there exists a constant $\delta>0$ such that
\begin{equation}\label{ht}
\left\|(\lambda I-A)z\right\|^2_X+\left\|B^*z\right\|^2_U\geq \delta
\left\|z\right\|^2_X,\quad\forall\,z\in D(A),\lambda\in\mathbb{R}.
\end{equation}
We shall refer to (\ref{ht}) as the (infinite-dimensional) Hautus
test.
\end{prp}
A new result in this paper reads as follows.

\begin{thm}\label{princ}
Let $A:D(A)\longrightarrow X,\,B\in\mathcal{L}(U,X)$ such that
$X,\,U,\,D(A)\subset X$ be complex Hilbert spaces and assume that
\begin{equation}\label{condb}
\quad \mathrm{Im}\, B\subset D(A),
\end{equation}
$$\mathcal{A}(\lambda)=\left[
\begin{array}{cc}
  \lambda-A & B \\
  B^* & 0 \\
\end{array}
\right]:D(A)\times U\longrightarrow X\times U,$$ and that $B^*$ has
a uniformly bounded right inverse.\\ If for $Q=\lambda-A,\quad
|\left|Q_{\mathrm{Im}\,B}|\right|_{\mathcal{L}(X)}=\mathcal{O}(1)$,
then
\begin{equation*}
\mathcal{A}(\lambda)\left[
\begin{array}{c}
  u \\
  u_- \\
\end{array}
\right]=\left[
\begin{array}{c}
  v \\
   v_+\\
\end{array}
\right]
\end{equation*}gives that
\begin{equation}\label{htest}
\|v\|^2_X+\|v_+\|^2_U\geq C(\|u\|^2_X+\|u_-\|^2_U).
\end{equation}

\end{thm}
\vskip 0.3cm
\begin{remark}\label{cerp} {\rm
\begin{enumerate}\vskip 0.1cm
\item In Theorem \ref{princ}, we remark that we don't need to have a
well-posed Grushin problem in order to get inequality (\ref{htest}).
\item Since $B^*$ has a uniformly bounded right inverse, then $B^*$
is surjective and according to N. K. Nikolski \cite{N}, the  system
$(A,B^*)$ is exactly controllable in any time $\tau>0$ with $A$ is
skew-adjoint operator.
\item From the point of view of applications it is often sufficient to have an explicit description of the space accessible states
$$\mathrm{Im}\,c(\tau)=c(\tau) L^2(0,\tau;U)$$
instead of strong demand restrictive exact controllability
$\mathrm{Im}\,c(\tau)=X$.
\item If $A$ is skew-adjoint on $X$ and $B\in\mathcal{L}(U,X)$ as
in the previous theorem and for
$\lambda=i\omega,\,\omega\in\mathbb{R}$, therefore  we have the
following Hautus type estimation
\begin{equation}\label{ha12}
\|(i\omega-A)u\|^2_X+\|B^*u\|^2_U\geq C\|u\|^2_X,\quad \forall\,u\in
X.
\end{equation}

and if we consider the following abstract control problem with
observation with $A$ and $B$ as above in the previous theorem
\begin{equation}\label{100}
\left\{
\begin{array}{ll}
\dot{z}(t)=Az(t)+Bu(t),\quad z(0)=0\\
y(t)=B^*z(t)\\
\end{array}
\right.
\end{equation}

then, (\ref{100}) is exactly observable and the $C_0$-groups of
isometries $T(t)_{t\in\mathbb{R}}$ generated by $A$ is exponentially
stable via estimation (\ref{ha12}). Hence, the setup of a Grushin
problem (not necessary well-posed) give us exactly observable
system.
\end{enumerate}}
\end{remark}
\begin{proof}of Theorem \ref{princ}.\vskip 0.1cm
Let $\Pi:X\longrightarrow (\mathrm{kernel}\,
B^*)^\perp=\mathrm{Im}\, B$ be the orthogonal projection. Then
\begin{equation*}
\begin{split}
 \|(I-\Pi)u\|^2_X&\leq |\langle P(I-\Pi)u,(I-\Pi)u\rangle|\\
&=|\langle(I-\Pi)v-(I-\Pi)P\Pi u, (I-\Pi)u\rangle|\\
&\leq\|v\|_X\|(I-\Pi)u\|_X+\|P\Pi u\|_X\|(I-\Pi)u\|_X,\\
\end{split}
\end{equation*}
 with $P=(\lambda I-A)$. By assumption, there exists a uniformly bounded operator
$$P_+:U\longrightarrow (\mathrm{ker}\,B^*)^\perp\subset X$$
such that $B^*P_+v_+=v_+$, and consequently $\Pi u=P_+v_+$. Thus
$$ \|(\lambda-A)\Pi
u\|_X=\|(\lambda-A)|_{\mathrm{Im}\,B}P_+v_+\|=\mathcal{O}(1)\|v_+\|_U,$$
and hence
$$\|(I-\Pi)u\|_X\leq \|v\|+\mathcal{O}(1)\|v_+\|_U.$$
With $P_-=P^*_+$, also we have $P_-B^*u_-=u_-$, so that
$$u_-=P_-(v-(\lambda-A)u)=P_-v-P_-\Pi(\lambda-A)(I-\Pi)u-P_-(\lambda-A)\Pi
P_+v_+$$ and
\begin{equation*}
\begin{split}
\|u\|_X&\leq
C(\|v\|+\|(\lambda-A)|_{\mathrm{Im}\,B}\|_{\mathcal{L}(X)}\|(I-\Pi)u\|_X+\|(\lambda-A)|_{\mathrm{Im}\,B}\|_{\mathcal{L}(X)}\|v_+\|_U\\
&\leq C(\|v\|_X+\|(I-\Pi)u\|_X+\|v_+\|_U.\\
\end{split}
\end{equation*}
It 's easy to prove that $\|\Pi u\|_X=\|P_+v_+\|\leq C\|v_+\|$ and
therefore
$$\|v\|_X+\|v_+\|_U\geq C(\|u\|_X+\|u_-\|_U).$$
\end{proof}

\begin{example}{\rm
We consider the following initial and boundary value problem:
\begin{equation}\label{ev1}
\partial^2_tu-\Delta u+G\partial_tu=0,\quad\Omega\times (0,+\infty),
\end{equation}
\begin{equation}\label{ev2}
u=0,\qquad\partial\Omega\times (0,+\infty),
\end{equation}
\begin{equation}\label{ev3}
u(.,0)=u^0\in H^2(\Omega)\cap
H^1_0(\Omega),\;\partial_tu(.,0)=u^1\in H^1_0(\Omega),\qquad\Omega
\end{equation}where $G=(-\Delta)^{-1}$. If we introduce the
following notations:
$$H=L^2(\Omega),\; \mathcal{D}(A_0)=H^2(\Omega)\cap H^1_0(\Omega)$$
$$A_0\varphi=-\Delta\varphi,\quad\forall\varphi\in\mathcal{D}(A_0).$$
The system (\ref{ev1})-(\ref{ev3}) can be written in the following
abstract form:
\begin{equation}\label{abs1}
\left\{
\begin{array}{ll}
\dot{z}(t)=\mathcal{A}_dz(t)\\
z(0)=z^0,\\
\end{array}\right.
\end{equation}where
$$z(t)=\left(
\begin{array}{c}
  u \\
  \partial_tu \\
\end{array}
\right),\quad \mathcal{A}_d=\left(
\begin{array}{cc}
  0 & I \\
   -\Delta& -G \\
\end{array}
\right),\quad z^0=\left(
\begin{array}{c}
   u^0\\
  u^1 \\
\end{array}
\right).$$ $A_d$ can be written in the form
$\mathcal{A}_d=\mathcal{A}_0+BB^*$ with
$$\mathcal{A}_0=\left(
\begin{array}{cc}
  0 & I \\
  -\Delta & 0 \\
\end{array}
\right),\qquad B=\left(
\begin{array}{c}
  0 \\
  G^{\frac{1}{2}} \\
\end{array}
\right),\qquad B^*=\left(
\begin{array}{cc}
  0 & G^{\frac{1}{2}} \\
\end{array}
\right).$$ Then, let
$$\mathcal{B}=BB^*:U=(H^2(\Omega)\cap
H^1_0(\Omega))\times H^1_0(\Omega)\longrightarrow (H^2(\Omega)\cap
H^1_0(\Omega))\times H^1_0(\Omega)$$ thus, it's easy to check that
that $B$ is onto and that the range of $\mathcal{B}$ is contained in
$D(\mathcal{A}_0)=(H^2(\Omega)\cap H^1_0(\Omega))\times
H^1_0(\Omega)$. Then according to the above Theorem, we have
\begin{equation*}
\exists
\delta'>0;\quad\left\|(i\omega-\mathcal{A}_0)\left(\begin{array}{c}z\\y\\\end{array}\right)\right\|^2_{X}+\left\|\mathcal{B}^*\left(\begin{array}{c}z\\y\\\end{array}\right)\right\|^2_U
\geq\delta'\left\|\left(\begin{array}{c}z\\y\\\end{array}\right)\right\|^2_{X},
\end{equation*}
$$\quad\forall\;\omega\in\mathbb{R},\;\left(\begin{array}{c}z\\y\\\end{array}\right)\in
D(\mathcal{A}_0).$$ which is equivalent to (with $y=iwz$)
\begin{equation}
\left\|(\omega^2-A_0)z\right\|^2_H+\left\|\omega
G^{\frac{1}{2}}z\right\|^2_U\geq \delta\|\omega
z\|^2,\quad\forall\omega\in\mathbb{R},\;z\in \mathcal{D}(A_0).
\end{equation}
}
\end{example}
\vskip 0.2cm
 Now, we introduce other types of controllability of
system (\ref{100}). Before that we recall the notion of Riesz basis
and for more details, we refer readers to \cite{TW}.
\begin{defi}
A sequence $(\varphi_n)_{n\geq 1}$ in a Hilbert space $X$ forms a
Riesz basis if
\begin{enumerate}
\item $\overline{\mathrm{span}}\{\varphi_n\}=X$ and
\item There exist positive constants $m$ and $M$ such that for an
arbitrary integer $n$ and scalar $(a_n)_{n\geq 1} $ one has
$$m\sum_{n\geq 1}|a_n|^2\|\varphi_n\|^2\leq \|\sum_{n\geq
1}a_n\varphi_n\|^2\leq M\sum_{n\geq 1}|a_n|^2\|\varphi_n\|^2.$$
\end{enumerate}
\end{defi}
From the definition, one can easily see that an orthonormal complete
sequence in a Hilbert space is a Riesz basis. Hence, Riesz basis is
such a basis that is equivalent to orthonormal basis under bounded
invertible transform, that's, for any given Riesz basis
$(\varphi_n)_{n\geq 1}$ in $X$, there exist a bounded invertible
operator $T$ such that
$$T\varphi_n=e_n,\quad n\geq 1$$
where $(e_n)_{n\geq 1}$ is an orthonormal basis. Also once we have a
Riesz basis $(\varphi_n)_{n\geq 1}$ for $X$, then we can identify
$X$ with $\ell^2$ via
$$x=\sum_{n\geq 1}a_n\varphi_n\in X\longleftrightarrow\,\sum_{n\geq
1}|a_n|^2<\infty.$$ As we said in Remark \ref{cerp} about the
characterization of $\mathrm{Im}\,c(t)$, the following theorem of
Nikolski \cite{N} gives an explicit description of
$\mathrm{Im}\,c(t)$ in the case where the generator $A$ has a Riesz
basis of eigenvectors.
\begin{thm}\label{rbasis}
Let $(\varphi_n)_{n\geq 1}$ be a Riesz basis in $X$ consisting of
eigenvectors of $A$ and $(\psi_n)_{n\geq 1}$ its biorthogonal and
assume that $$A\varphi_n=-\lambda_n\varphi_n,\quad n\geq 1$$ then,
if the family $(\mathcal{E})_{n\geq 1}$ defined by
$$\mathcal{E}_n(t)=e^{-\bar{\lambda_n}t}B^*\psi_n$$
is also a Riesz basis in $L^2(0,t;U)$  then
$$\mathrm{Im}\,c(t)=\{\sum_{n\geq
1}b_n\varphi_n\;,\sum_{n\geq1}|b_n|^2\frac{1}{\|\mathcal{E}_n\|^2_{L^2}}\,<\infty\}.$$
\end{thm}
\vskip 0.15cm In the case where $A$ is as in the previous theorem
then, each state $x\in X$ is defined formally by its Fourier series
$$x\sim \sum_{n\geq 1}<x,\psi_n>\varphi_n$$
where $(\psi_n)_{n\geq 1}$ is the biorthogonal sequence. It's
natural to search an explicit description of the control space
$\mathrm{Im}\,c(t)$ in the form of "Fourier multipliers".
\begin{defi}
Let $(\omega_n)_{n\geq 1}$ be a positive sequence of reel number. We
put
$$X(\omega_n)=\{x\in X/\,\exists\,y\in
X\,\text{s.t}\,<x,\psi_n>=\frac{1}{\omega_n}<y,\psi_n>,\,n\geq
1\}.$$The system $(A,B)$ is said to be exactly controllable in time
$T>0$ up to a renomalization if there exist $\omega_n>0,\,n\geq 1$
such that
$$S(t)X(\omega_n)\subset X(\omega_n);\quad t\geq 0$$
\begin{equation}\label{imb}
BU=\mathrm{Im}\,B\subset X(\omega_n)
\end{equation}
and $(A|_{X(\omega_n)},B)$ is exactly controllable.
\end{defi}
\vskip 0.2cm The following proposition link the Hautus test criteria
obtained in Theorem \ref{princ} with condition (\ref{imb})
introduced in te obove definition of controllability up to a
renormalization.
\begin{prp}
Let $A:D(A)\longrightarrow X$ as in Theorem \ref{rbasis} and
$B\in\mathcal{L}(U,X)$.\\
 With the renormalization
\begin{equation}\label{renrm}
X(\omega_n)=\{\sum_{n\geq 1}a_n\varphi_n;\quad \sum_{n\geq
1}|a_n|^2\frac{1}{\|\mathcal{E}_n\|^2_{L^2}}<\infty\}
\end{equation}
we assume that we have
\begin{equation}\label{condb1}
\quad \mathrm{Im}\, B\subset X(\omega_n),
\end{equation}
$$\mathcal{A}(\lambda)=\left[
\begin{array}{cc}
  \lambda-A & B \\
  B^* & 0 \\
\end{array}
\right]:X(\omega_n)\times U\longrightarrow X\times U,$$ and that
$B^*$ has a uniformly bounded right inverse.\\ If for
$Q=\lambda-A,\quad
|\left|Q_{\mathrm{Im}\,B}|\right|_{\mathcal{L}(X)}=\mathcal{O}(1)$,
then
\begin{equation*}
\mathcal{A}(\lambda)\left[
\begin{array}{c}
  u \\
  u_- \\
\end{array}
\right]=\left[
\begin{array}{c}
  v \\
   v_+\\
\end{array}
\right]
\end{equation*}gives that
\begin{equation}\label{htest11}
\|v\|^2_X+\|v_+\|^2_U\geq C(\|u\|^2_X+\|u_-\|^2_U).
\end{equation}

\end{prp}
\vskip 0.15cm
\begin{proof}
The proof is the same as Theorem \ref{princ}.
\end{proof}

\end{document}